\documentclass[12pt]{amsart}

\usepackage{breakurl}
\usepackage{graphicx}
\usepackage{amsmath}
\usepackage{amscd}
\usepackage{amsfonts}
\usepackage{amssymb}
\usepackage[dvipsnames]{xcolor}
\usepackage{fullpage}
\usepackage{mathtools}
\usepackage[pagebackref,hypertexnames=false, colorlinks, citecolor=red,linkcolor=blue, urlcolor=red]{hyperref}
\usepackage{todonotes}
\usepackage{tikz-cd}
\usepackage{enumitem}
\usepackage{amsthm}

\numberwithin{equation}{section}

\newtheorem{theorem}{Theorem}[section]
\newtheorem{lemma}[theorem]{Lemma}
\newtheorem{proposition}[theorem]{Proposition}

\theoremstyle{definition}
\newtheorem{example}[theorem]{Example}
\newtheorem{remark}[theorem]{Remark}
\newtheorem{definition}[theorem]{Definition}
\newtheorem{problem}[theorem]{Problem}

\newcommand{\be}{\begin{equation}}
\newcommand{\ee}{\end{equation}}
\newcommand{\bes}{\begin{equation*}}
\newcommand{\ees}{\end{equation*}}

\newcommand{\cA}{\mathcal{A}}
\newcommand{\cB}{\mathcal{B}}

\newcommand{\cE}{\mathcal{E}}
\newcommand{\cF}{\mathcal{F}}

\newcommand{\cH}{\mathcal{H}}

\newcommand{\cM}{\mathcal{M}}

\newcommand{\bB}{\mathbb{B}}
\newcommand{\bC}{\mathbb{C}}
\newcommand{\bD}{\mathbb{D}}

\newcommand{\bF}{\mathbb{F}}
\newcommand{\bM}{\mathbb{M}}
\newcommand{\bN}{\mathbb{N}}

\newcommand{\bS}{\mathbb{S}}
\newcommand{\bT}{\mathbb{T}}

\newcommand{\ol}{\overline}

\newcommand{\alg}{\operatorname{alg}}

\newcommand{\fin}{\mathrm{fin}}

\newcommand{\fA}{{\mathfrak{A}}}
\newcommand{\fB}{{\mathfrak{B}}}

\newcommand{\fV}{{\mathfrak{V}}}

\begin{document}

\title{Tensor algebras of subproduct systems and noncommutative function theory}

\author{Michael Hartz}
\thanks{M.\ Hartz was partially supported by a GIF grant and by the Emmy Noether Program of the German Research Foundation (DFG Grant 466012782).}
\address{M.H., Fachrichtung Mathematik, Universit\"at des Saarlandes, 66123 Saarbr\"ucken, Germany}
\email{hartz@math.uni-sb.de}

 \author{Orr Moshe Shalit}
 \address{O.S., Faculty of Mathematics\\
 Technion - Israel Institute of Technology\\
 Haifa\; 3200003\\
 Israel}
 \email{oshalit@technion.ac.il}
 \urladdr{\href{https://oshalit.net.technion.ac.il/}{\url{https://oshalit.net.technion.ac.il/}}}

 \thanks{O.M. Shalit is partially supported by ISF Grant no. 431/20.
 }

 \addcontentsline{toc}{section}{Abstract}

\begin{abstract}
We revisit tensor algebras of subproduct systems with Hilbert space fibers, resolving some open questions in the case of infinite dimensional fibers.
We characterize when a tensor algebra can be identified as the algebra of uniformly continuous noncommutative functions on a noncommutative homogeneous variety or, equivalently, when it is residually finite dimensional: this happens precisely when the closed homogeneous ideal associated to the subproduct system satisfies a Nullstellensatz with respect to the algebra of uniformly continuous noncommutative functions on the noncommutative closed unit ball. 
We show that --- in contrast to the finite dimensional case --- in the case of infinite dimensional fibers this Nullstellensatz may fail.
Finally, we also resolve the isomorphism problem for tensor algebras of subproduct systems: two such tensor algebras are (isometrically) isomorphic if and only if their subproduct systems are isomorphic in an appropriate sense. 
\end{abstract}

\maketitle

\section{Introduction and background}\label{sec:introduction}

\subsection{Subproduct systems and their tensor algebras}

A subproduct system over a monoid $\bS$ is a family $X = \left(X_s \right)_{s \in \bS}$ where all the $X_s$ are Hilbert C*-correspondences over the same C*-algebra $A$, that behave nicely with respect to tensor products. 
The notion was first formally isolated and defined in \cite{ShaSol09} and \cite{BM10} for the purpose of studying semigroups of completely positive maps, after decades in which subproduct system-like structures appeared implicitly in noncommutative dynamics (see \cite{ShaSke23} for the general definition and for some context). 
Besides their important role in dilation theory of CP-semigroups \cite{ShaSol09,ShaSke23,Ver16}, it was clear right from the outset and has been continuously recognized that subproduct systems provide a unified and flexible framework for treating a very wide variety of previously studied operator theoretic and operator algebraic structures; see, e.g., \cite{AK21,CDOHLZ19,DRS11,DorMar14,DorMar15,GerSke20,Gur12,HabNes21,KakSha19,SSS18,SSS20,STV16,Vis11,Vis12}.

In this paper, we will concentrate on the simplest kind of subproduct system, where the monoid $\bS$ is $\bN = \{0,1,2,\ldots\}$ and the C*-algebra $A$ is just the scalar field $\bC$. 

\begin{definition}\label{def:subps}
A {\em subproduct system} is a family $X = \left(X_n\right)_{n=0}^\infty$ of Hilbert spaces such that $X_0 = \bC$ and 
\be\label{eq:subps}
X_{m+n} \subseteq X_m \otimes X_n . 
\ee
\end{definition}
Even this seemingly restricted setting provides the basis for the formation of a plethora of operator algebras, ranging from the continuous functions on the unit sphere \cite{Arv98}, to algebras of analytic functions on homogeneous subvarieties of the unit ball \cite{DRS11}, to the Cuntz or Cuntz-Krieger C*-algebras \cite{KakSha19}, and so on.
Our interest is the simplest kind of operator algebra that one can attach to a subproduct system --- the tensor algebra.
 
If $X$ is a subproduct system, we find from the defining relation \eqref{eq:subps} of subproduct systems that $X_n \subseteq X_1^{\otimes n}$ for all $n$. 
We let $p_n \colon X_1^{\otimes n} \to X_n$ denote the corresponding orthogonal projection. 
If we need to specify the subproduct system $X$, then we shall use $p_n^X$ to denote this projection.

\begin{definition}\label{def:tensor}
Given a subproduct system $X$, the {\em $X$-Fock space} is defined to be the Hilbert space direct sum
\[
\cF_X = \bigoplus_{n=0}^\infty X_n .
\]
Every $x \in X_1$ defines a (left) {\em shift operator} $S(x)$ that is determined from its action on the summands: 
\[
S(x) \colon y \in X_n  \mapsto p_{n+1} (x \otimes y) \in X_{n+1}. 
\]
Finally, the {\em tensor algebra} of $X$ is defined to be the unital operator algebra generated by the shifts
\be\label{eq:tensoralg}
\cA_X = \ol{\alg \{ S(x) \big| x \in X_1\}} \subset B(\cF_X) . 
\ee
\end{definition}
We can extend the definition of $S$ to every $x \in X_m$ by 
\[
S(x) \colon y \in X_n  \mapsto p_{m+n} (x \otimes y) \in X_{m+n}. 
\]

The $X$-Fock space has a natural grading and this induces a grading on $\cA_X$: an element $T \in \cA_X$ is said to be {\em $n$-homogeneous} if $T (X_m) \subseteq X_{m+n}$ for all $m \in \bN$.
We let $\cA_X^{(n)}$ denote the $n$-homogeneous elements of $\cA_X$. 
The graded structure of $\cA_X$ plays a key role in the analysis of tensor algebras.
We shall require the following result.

\begin{proposition}[Proposition 9.3, \cite{ShaSol09}; Proposition 6.2 \cite{DorMar14}]\label{prop:grading}
For every $n \in \bN$, there is a completely contractive surjective idempotent $\Phi_n : \cA_X \to \cA_X^{(n)}$. 
Every $T \in \mathcal{A}_X$ can be written uniquely as a Ces\`aro norm-convergent sum $T = \sum_{n=0}^\infty T_n$ where $T_n = \Phi_n(T) \in \cA_X^{(n)}$ is $n$-homogeneous. 
Moreover, for all n, the map $S \colon X_n \to \cA_X$ given by $x \mapsto S(x)$ is an isometric surjection of $X_n$ onto $\cA_X^{(n)}$. 
\end{proposition}

By Proposition \ref{prop:grading}, $S$ is an isometric map from $X_m$ into $\cA_X$, and it follows that we can describe the tensor algebra also as
\[
\cA_X = \ol{\alg \{ S(x) \big| n \in \bN, x \in X_n\}}. 
\]

\begin{example}\label{ex:product}
Let $H$ be a Hilbert space, and consider the case in which $X_n  = H^{\otimes n}$ for all $n$. 
In this case the $X$-Fock space $\cF_X$ is simply the free Fock space $\cF(H) = \oplus_n H^{\otimes n}$ and the tensor algebra $\cA_X$ is the norm closed algebra generated by the row isometry $L=(L_i:=S(e_i))_{i \in \Lambda}$ for some orthonormal basis $\{e_i\}_{i\in \lambda}$, which is Popescu's noncommutative disc algebra $\fA_d$, for $d = |\Lambda| = \dim H$ \cite{Pop96}. 
In particular, when $H = \bC$, then $X_n  =\bC$ for all $n$.
In this case $\cF_X$ is easily identifiable with $\ell^2(\bN)$, and for $x = 1_\bC$ the shift $S(x)$ is then just the unilateral shift. 
The tensor algebra $\cA_X$ in this case is simply the disc algebra $A(\bD)$. 
\end{example}

\begin{example}\label{ex:sym}
Let $H$ be a Hilbert space, and consider the case in which $X_n  =$ the symmetric product of $H$ with itself $n$-times.
In this case the $X$-Fock space $\cF_X$ is the symmetric Fock space over $H$, which is also known as the Drury-Arveson space $H^2$ and has a natural interpretation as a reproducing kernel Hilbert space \cite{Arv98}. 
The tensor algebra $\cA_X$ is the norm closed algebra generated by the polynomials inside the multiplier algebra of $H^2$.
\end{example}

\subsection{The goals of this paper, briefly}
The purpose of this paper is to collect what is known about the representations and the classification of tensor algebras of subproduct systems of Hilbert spaces over $\bN$ and to complete the missing pieces. 
In particular, we sought to understand \emph{(a) when are two tensor algebras (isometrically) isomorphic}, and \emph{(b) can all tensor algebras be faithfully represented as an algebra of continuous noncommutative (nc) functions on a homogeneous nc variety}.
It is worth noting that although all the aspects that we treat have already been fully solved in the case of finite dimensional fibers (i.e., when $\dim X_1 < \infty$), in the case of infinite dimensional fibers some aspects have been left open, due to nontrivial technical complications. 
Indeed, results in several complex variables or affine  algebraic geometry, which played a role in the solution of the problems in the finite dimensional case, are not readily applicable in the infinite dimensional case. 
But not only do the proofs require modification as we pass to infinite dimensions, some of the results are different, too.

We shall show that the subproduct systems constitute a complete invariant for the tensor algebras, both in the category of tensor algebras with isometric isomorphism, as well as in the category of tensor algebras with completely bounded isomorphisms. 
With every subproduct system $X$ we associate a closed homogeneous ideal $I_X$ in $\fA_d$ for which $\cA_X \cong \fA_d / I_X$. 
We examine a Nullstellensatz (in free variables) for such ideals, and we show that the ideal $I_X$ satisfies this Nullstellensatz precisely when $\cA_X$ can be identified with the algebra of uniformly continuous nc functions on the nc variety consisting of all matrix row contractions that are annihilated by $I_X$. 
We provide an example showing that unlike in the case of finitely many variables, in the case of infinitely many variables this Nullstellensatz may fail, thus the tensor algebras are {\it not} in one-to-one correspondence with algebras of uniformly continuous nc analytic functions on nc homogeneous varieties, they need not be residually finite dimensional, and form a richer class of operator algebras than we may have thought.

\subsection{The isomorphism problem for tensor algebras}
A natural question that arises is: to what extent does a subproduct system determine its tensor algebra? 
To make this question precise we will need a few more definitions, but a few words are in order before the definitions.
First, \cite[Example 5.1]{Gur12} exhibits two non-isomorphic subproduct systems of finite dimensional Hilbert spaces over $\bN \times \bN$ that give rise to the same tensor algebra.  
Thus, it makes sense to work over $\bN$ rather than more general semigroups, since over $\bN$ there is still hope that subproduct systems are a complete invariant. 
On the other hand, if we were to work with subproduct systems over $\bN$ whose fibers are general Hilbert C*-correspondences rather than Hilbert spaces, then the question of whether a tensor algebra is uniquely determined by its subproduct system  is a long standing open problem since \cite{MuhSol00} even in the case of a {\em product system}, that is, the case where we have equality in \eqref{def:subps}. 
In this case the tensor algebra is Muhly and Solel's {\em tensor algebra of a C*-correspondence} \cite{MuhSol98} --- classification of these tensor algebras contains the longstanding isomorphism problem for semicrossed products of C*-dynamical systems {\em as a special case}; a solution for the {\em isometric} isomorphism problem in this context has appeared only quite recently \cite{KatRam+} (see also \cite{DavKak14}). 
Finally, as we shall recall below, the problem is completely solved in the case of finite dimensional Hilbert space fibers, thus the case of general Hilbert spaces seems like the most urgent open case in the broad isomorphism problem of tensor algebras.

%

\begin{definition}\label{def:subps_similar}
Let $X = \left(X_n\right)_{n=0}^\infty$ and $Y = \left(Y_n\right)_{n=0}^\infty$ be two subproduct systems. 
We say that $X$ and $Y$ are {\em similar} if there exists a sequence $V = (V_n)_{n=0}^\infty$ of invertible linear maps $V_n \colon X_n \to Y_n$ such that $\sup_n \left\|V_n \right\| < \infty$, $\sup_n \left\|V_n^{-1}\right\| < \infty$ and
\be\label{eq:similarity}
\quad V_{m+n} (p^X_{m+n}(x \otimes y)) = p^Y_{m+n}(V_m x \otimes V_n y)
\ee
for all $m,n$ and all $x \in X_m, y \in X_n$.
$V = (V_n)_{n=0}^\infty$ is then said to be a \emph{similarity}.
In case that $V_n$ is a unitary for all $n$, then $V$ is said to be an {\em isomorphism}. 
Two subproduct systems are said to be {\em isomorphic} if there is an isomorphism between them.
\end{definition}


If we have a similarity $V \colon X \to Y$ between two subproduct systems, then one may construct a bounded invertible map
\[
W := \bigoplus_{n=0}^\infty V_n \colon \cF_X =
 \bigoplus_{n=0}^\infty X_n \to \cF_Y = \bigoplus_{n=0}^\infty Y_n , 
\]
and it is then not hard to check that this gives rise to a completely bounded isomorphism $\varphi : \cA_X \to \cA_Y$ given by 
\[
\varphi(T) = W T W^{-1}.
\]
Of course, if $V$ is an isomorphism and not merely a similarity, then $W$ is a unitary and hence $\varphi$ is a completely isometric isomorphism. 
One of the main open problems that we shall address is whether the converse holds: 

\begin{problem}\label{prob:iso}
Suppose that $X$ and $Y$ are subproduct systems and that their tensor algebras $\cA_X$ and $\cA_Y$ are isomorphic in a certain sense: algebraically isomorphic, (completely) boundedly isomorphic or (completely) isometrically isomorphic. 
Does it follow that $X$ and $Y$ are similar/isomorphic?
\end{problem}

We note that with a tensor algebra one can also associate its maximal ideal space, or its space of finite dimensional representations, which form a variety or a noncommutative (nc) variety, respectively, which can serve as an invariant. 
This was the approach taken in \cite{DRS11,KakSha19,SSS18,SSS20}. 
However here we return to the most basic classification problem, Problem \ref{prob:iso}, as considered in \cite{ShaSol09,DorMar14}, in order to classify tensor algebras in terms of their subproduct systems. 
In fact, we shall show that the nc varieties in general {\em do not} classify tensor algebras.

\subsection{The main results}

Our main result on classification of tensor algebras is the following.

\begin{theorem}\label{thm:TFAE}
Let $X$ and $Y$ be subproduct systems. 
The following are equivalent: 
\begin{enumerate}
\item There exists a bounded isomorphism $\varphi \colon \cA_X \to \cA_Y$. 
\item There exists a completely bounded isomorphism $\varphi \colon \cA_X \to \cA_Y$. 
\item There exists a similarity $W \colon X \to Y$. 
\end{enumerate}
Also, the following are equivalent: 
\begin{enumerate}
\item There exists an isometric isomorphism $\varphi \colon \cA_X \to \cA_Y$. 
\item There exists a completely isometric isomorphism $\varphi \colon \cA_X \to \cA_Y$. 
\item There exists an isomorphism $W \colon X \to Y$. 
\end{enumerate}
\end{theorem}

After several earlier partial versions, this theorem was obtained in \cite[Theorem 3.4]{KakSha19} in the case where $\dim X_1 < \infty$. 
It turns out that the key issue is to show that the existence of a bounded isomorphism implies the existence of a {\em vacuum preserving} isomorphism; this reduction was used in several works and received definite form in \cite{DorMar14}. 
We shall explain this reduction in Section \ref{subsec:reduction}, and settle the isomorphism problem in the remainder of Section \ref{sec:classification}.

In Section \ref{sec:representation} we determine all bounded finite dimensional representations of tensor algebras and, in particular, we describe their maximal ideal spaces.
Section \ref{sec:representation} also contains the second main result obtained in this work: a characterization of when a tensor algebra $\cA_X$ can be identified as an algebra of uniformly continuous nc holomorphic functions on a homogeneous nc holomorphic variety, via an analysis of the finite dimensional representations of $\cA_X$. 
With every subproduct system $X$ we identify a norm closed ideal $J = J_X \triangleleft \,\fA_d$ (where $d = \dim X_1$). 
With this ideal we identify an nc variety $V(J)$, which consists of all contractive $d$-tuples of $n \times n$ matrices ($n \in \bN$) that are annihilated by elements of $J$. 
This nc variety can be identified with the space of completely contractive, finite dimensional representations of $\cA_X$, and every element of $\cA_X$ determines an nc function on $V(J)$. 
We then show that $\cA_X$ is residually finite dimensional (i.e., it is completely normed by its finite dimensional representations), if and only if the restriction map from $\cA_X$ to $V(J)$ is injective, and that this happens if and only if the following Nullstellensatz holds:
\[
J = I(V(J)), 
\]
where for an nc set $\Omega$ we let $I(\Omega)$ denote the ideal of all elements in $\fA_d$ that vanish on $\Omega$. 
In the case $\dim X_1 < \infty$ it was shown that this Nullstellensatz holds for every closed and homogeneous ideal, and that tensor algebras can always be considered as nc functions on an nc variety (see Sections 9 and 10 in \cite{SSS18}). 
However, we show that when $\dim X_1 = \aleph_0$ the Nullstellensatz holds for certain ideals and fails for others. 
To state these results precisely we require some preliminaries on nc functions which are discussed in the following section.
 
 \vskip 10pt

\noindent{\bf Acknowledgements.} 
We wish to thank the anonymous referee, whose questions and suggestions made this a better paper.

\section{Noncommutative functions and subproduct systems}\label{sec:functions}

\subsection{Noncommutative functions and the noncommutative unit ball}
We now recall the rudiments of the theory of noncommutative (nc) functions that are relevant to our main problem. 
The basic general theory is developed in \cite{KVV14} and a very rapid introduction to the subject can be found in \cite{AM16}; see also \cite{SSS18} for an introduction that is geared toward the kind of non-selfadjoint operator algebras that we are studying.
Let $H$ be a Hilbert space of dimension $d$. 
We fix throughout an index set $\Lambda$ with $|\Lambda| = d$ and an orthonormal basis $\{e_i\}_{i \in \Lambda}$ for $H$. 
For us, the most interesting case is $d = \aleph_0$. 
We note that in the literature the case $d > \aleph_0$ is hardly considered, but at least some of the basics can be developed in this generality as well (in essence, each individual function in $d>\aleph_0$ variables will only involve at most $\aleph_0$ variables). 

\begin{remark}
Before going into the definitions, let us give a word of motivation addressed at the reader who is wary of studying analytic functions in infinitely many variables. 
Besides the intrinsic interest, one good reason to study algebras in infinitely many variables is the hope that they may give a concrete representation of a large class of operator algebras; this is one of our concerns in this paper. 
Another good reason to study analytic functions in infinitely many variables, is that even Hilbert function spaces in a single complex variable may have useful representations as spaces of functions in infinitely many variables. 
Agler and McCarthy proved that every complete Pick reproducing kernel Hilbert space can be identified with the restriction of the Drury-Arveson space in $d$ variables to a subvariety of the unit ball $\bB_d$ \cite{AM00}. 
In general, their theorem requires $d = \infty$. 
Indeed, $d = \infty$ occurs naturally in uncountably many typical examples of function spaces on the unit disc \cite{DHS15},
and $d = \infty$ is in fact necessary for many classical spaces on the unit disc, including the Dirichlet space; see \cite{Roc17} and \cite[Corollary 11.9]{Hartz17a}.
In the case of the Dirichlet space, this phenomenon persists even if we weaken the notion of identification \cite{Hartz22}.
\end{remark}

For every $n$, we let $M_n(\bC)^d := M_n(\bC)^\Lambda$ denote the space of all $d$-tuples $T = (T_i)_{i\in \Lambda}$ that define a bounded operator from the direct sum $\bigoplus_{i\in \Lambda} \bC^n$ to $\bC^n$. 
We let $\|T\| = \|\sum_i T_iT_i^*\|$ denote the norm of this operator, referred to also as the row norm of the tuple $T$. 
The set $\bM_d = \sqcup_n M_n(\bC)^d$ is our {\em noncommutative universe}, by which we mean the big set that contains all the ``things" that we will plug as arguments into nc functions.
Elements of $\bM_d$ can be considered either as $d$-tuples $T = (T_i)_{i \in \Lambda}$ or as (convergent) sums $\sum_i T_i \otimes e_i$ in the $n \times n$ matrices over the row operator space $H_r$.
On $\bM_d$ there are natural operations of direct sum $T\oplus U = (T_i \oplus U_i)_{i \in \Lambda}$ and left/right multiplication by matrices: $S \cdot (T_i)_{i \in \Lambda} = (ST_i)_{i \in \Lambda}$ and $T \cdot S = (T_i S)_{i \in \Lambda}$.

\begin{definition}\label{def:nc_function}
Let $\Omega = \sqcup_n \Omega_n$ be an {\em nc set}, that is, a graded subset of $\bM_d$, where $\Omega_n \subseteq M_n(\bC)^d$ for all $n$, that is closed under direct sums.
A {\em nc holomorphic function} is a map $f : \Omega \to \bM_1$ that satisfies 
\begin{enumerate}
\item $f$ is graded: $f(\Omega_n) \subseteq M_n(\bC)$, 
\item $f$ respects direct sums: $f(X \oplus Y) = f(X) \oplus f(Y)$ for all $X,Y \in \Omega$,
\item $f$ respects similarity: $f(S^{-1}\cdot X \cdot S) = S^{-1} \cdot f(X) \cdot S$ whenever $X \in \Omega$ and $S^{-1} \cdot X \cdot S \in \Omega$.
\end{enumerate}
\end{definition}

For us, the most interesting nc set will be the {\em noncommutative (nc) open unit ball} $\fB_d$ defined by 
\[
\fB_d = \bigsqcup_{n=1}^\infty  \left\{(T_i)_{i \in \Lambda} \in M_n(\bC)^d : \left \|\sum_{i \in \Lambda} T_i T_i^* \right\| < 1\right\}. 
\]
In other words, $\fB_d$ is the set of all strict row contractions consisting of matrices. 

On $\fB_d$ we have the natural tuple of coordinate functions $z = (z_i)_{i\in \Lambda}$ defined by $z_j(T) = T_j$. 
The functions $z_i, i\in \Lambda$, together with the unit $1$ generate the algebra $\bF_d$ of free polynomials in $d$ noncommuting variables. 
When treating finite $d$, ideals in the algebra $\bF_d$ were enough to determine all subproduct systems $X$ with $\dim X_1 = d$. 
However, when $d = \infty$, the ideal structure of $\bF_d$ is not rich enough, and we will need to work with function algebras. 
The issue is, for example, that a function of the form $f(z) = \sum_{k=1}^\infty a_k z_k$ for a given sequence $(a_k) \in \ell^2$ is not a polynomial according to the above definition, but it is a homogeneous nc function of degree one that is contained in either one of the Banach algebras of nc functions that we consider below (some authors do indeed consider bounded linear functionals and their products as polynomials). 

The {\em nc Drury-Arveson space} $\cH_d^2$ is defined to be the space of all nc functions $f$ on $\fB_d$ that have a Taylor series $\sum_\alpha a_\alpha z^\alpha$ for which $\|f\|^2_{\cH^2_d} = \sum |a_\alpha|^2 < \infty$. 

The algebra of all bounded nc holomorphic functions on $\fB_d$ is denoted $H^\infty(\fB_d)$; this algebra is a Banach algebra with respect to the supremum norm $\|f\|_\infty := \sup\{\|f(X)\| : X \in \fB_d\}$. 
It can be shown that $H^\infty(\fB_d)$ is in fact equal to the algebra of left multipliers of $\cH^2_d$; see \cite{SSS18} for details.

\begin{definition}\label{def:nc_cont_function}
If $\Omega \subseteq \fB_d$ is an nc set, we say that an nc function $f : \Omega \to \bM_1$ is {\em uniformly continuous} if for all $\epsilon > 0$, there is a $\delta > 0$, such that for all $n$ and all $X,Y \in \Omega_n$, we have that $\|X - Y\|<\delta$ implies $\|f(X) - f(Y)\|<\epsilon$. 
We let $A(\Omega)$ denote the algebra of all uniformly continuous nc functions on $\Omega$. 
In particular, $A(\fB_d)$ denotes the algebra of all uniformly continuous nc functions on $\fB_d$. 
We shall denote $A_d = A(\fB_d)$. 
\end{definition}
Every element in $A_d$ extends uniquely to a uniformly continuous nc function on $\ol{\fB}_d$ (defined as the levelwise closure of $\fB_d$, which is equal to the set of all row contractions acting on finite dimensional spaces). 

In \cite[Section 9]{SSS18}, it was shown that $A_d$ is equal to the norm closure of the polynomials with respect to the sup norm in $H^\infty(\fB_d)$ (although there it was implicitly assumed that $d\leq \aleph_0$, the argument works for any $d$). 
This norm closure of polynomials can be identified with the tensor algebra $\fA_d$ corresponding to the full product system (see Example \ref{ex:product}), by the natural map that sends the coordinate function $z_i$ to the operator $S(e_i)$; see \cite[Section 3]{SSS18}. 
We shall henceforth identify $\fA_d$ with $A_d$ and identify elements in $\fA_d$ as nc functions in $A_d$ and vice versa, as convenient. 
By Proposition \ref{prop:grading}, we can also identify $A_d$ as a linear and dense subspace of the full Fock space $\cF(H)$ for a $d$-dimensional Hilbert space $H$, via
\[
H^{\otimes n} \ni x \mapsto S(x) \in \fA_d \cong A_d
\]
and 
\be\label{eq:identification}
\fA_d \cong A_d \ni f \mapsto f((S(e_i))_{i\in \Lambda}) 1 \cong f\cdot 1 \in \cF(H). 
\ee

\subsection{Homogeneous ideals and varieties}

Every $f \in \cH^2_d$ can be written as a norm convergent sum $f = \sum f_n$ where $f_n$ is the norm limit of $n$-homogeneous polynomials. 
The function $f_n$ is said to be an {\em $n$-homogeneous function}. We shall let $\cH^2_d(n)$ denote the space of all $n$-homogeneous functions in $\cH^2_d$. 
Every function $f \in A_d$ has a Ces\`aro convergent decomposition 
\[
f = \sum f_n 
\]
into {\em homogeneous components}; see \cite[Lemma 7.9]{SSS18}. 
The $n$-homogeneous component $f_n$ can be obtained from $f$ by the completely contractive projection
\[
f_n(X) = \phi_n(f)(X) = \frac{1}{2\pi} \int_0^{2\pi} f(e^{i\theta} X)e^{-in\theta} d\theta  \,\, , \,\, X \in \fB_d.
\]
Note that under the identification $A_d \cong \fA_d$ the projection $\phi_n$ corresponds to $\Phi_n$ from Proposition \ref{prop:grading}.

\begin{definition}\label{def:hom_ideal}
An ideal $J \subseteq A_d$ is said to be {\em homogeneous}, if whenever $f \in J$, then for all $n$, the homogeneous component $f_n$ is also in $J$. 
A subset $\fV \subseteq \fB_d$ is said to be a {\em homogeneous variety} if it is the joint zero locus (in $\fB_d$) of all functions in a homogeneous ideal $J$: 
\[
\fV = V(J) := \{X \in \fB_d : f(X) = 0 \,\, \textrm{ for all } \,\, f \in J\}.
\]
\end{definition}

For any subset $\Omega \subseteq \fB_d$, we denote 
\[
I(\Omega) := \{f \in A_d : f(X) = 0 \,\, \textrm{ for all } \,\, X \in \Omega \}. 
\]
One can show that if $\fV \subseteq \fB_d$ is a homogeneous variety, then it is a {\em homogeneous set} in the sense that $\bC X \cap \fB_d \subseteq \fV$ for all $X \in \fV$, and from this it follows that $I(\fV)$ is a homogeneous ideal.

By \cite[Section 9]{SSS18} for every homogeneous nc variety $\fV \subseteq \fB_d$, the algebra $A(\fV)$ is the closure in the supremum norm of the polynomials, i.e., the closure of the algebra generated by the restriction of coordinate functions to $\fV$.
Note that every nc function in $A(\fV)$ extends to an nc function on $\ol{\fV}$ and we may identify $A(\fV)$ with $A(\ol{\fV})$. 

The operators $I(\cdot)$ and $V(\cdot)$ are inclusion reversing maps between closed ideals and nc varieties and vice versa, but they are not mutual inverses. 
We have the following tautological fact that is well known in many analogous circumstances.
\begin{lemma}\label{lem:triple_tautology}
For every $\Omega \subseteq \fB_d$ and every $E \subseteq A_d$, we have 
\[
I(\Omega) = I(V(I(\Omega))
\]
and
\[
V(I(V(E))) = V(E).
\]
\end{lemma}
\begin{proof}
This is a general fact true for any pair of such order reversing operators, so we explain just one equality. 
Clearly $I(\Omega) \subseteq I(V(I(\Omega)))$. 
But for the same reason $\Omega \subseteq V(I(\Omega))$, and so applying the inclusion reversing $I(\cdot)$ we obtain $I(\Omega) \supseteq I(V(I(\Omega)))$.
\end{proof}

\begin{definition}\label{def:nullstz}
Let $J$ be a closed ideal in $\fA_d$. 
We say that $J$ {\em satisfies the Nullstellensatz} if 
\[
J = I(V(J)) = \{f \in \fA_d : f(X) = 0 \text{ for all } X \in V(J)\}.
\]
\end{definition}
By Lemma \ref{lem:triple_tautology}, every ideal that is the annihilating ideal of a prescribed zero set satisfies the Nullstellensatz. 
In particular, for every homogeneous $J \triangleleft A_d$, the ideal $I(V(J))$ is a closed homogeneous ideal in $A_d$ that satisfies the Nullstellensatz. 
On the other hand, we shall see in Section \ref{subsec:failure} that there are closed homogeneous ideals that do not satisfy the Nullstellensatz.

\subsection{Subproduct systems and closed homgeneous ideals}\label{subsec:ideals}

Given $d < \infty$ and an orthonormal basis $\{e_1, e_2, \ldots, e_d\}$ for a $d$-dimensional Hilbert space $H$, there is a $1$-$1$ correspondence between homogeneous ideals $I \triangleleft \bF_d$ and subproduct systems $X = (X_n)_{n=0}^\infty$ with $X_1 \subseteq H$ given by 
\[
I \longleftrightarrow X^I
\]
where the $n$th fiber of $X^I$ is given by 
\[
X^I_n = X_1^{\otimes n} \ominus \{ p(e) : p \in I \textrm{ is $n$-homogeneous}\}
\]
and $p(e) = \sum_{|\alpha|=n} a_\alpha e_{\alpha_1} \otimes \cdots \otimes e_{\alpha_n}$ where $p(z) = \sum_{|\alpha|=n} a_\alpha z_{\alpha_1} \cdots z_{\alpha_d}$ (see \cite[Proposition 7.2]{ShaSol09}). 
When $d = \infty$, this correspondence  ceases to be surjective. 
However, if we replace homogeneous ideals in $\bF_d$ with closed homogeneous ideals in $A_d$ we recover a bijective correspondence, as we will show below.

Letting now $H$ be a Hilbert space of dimension $d$, and given a homogeneous closed ideal $J \subseteq A_d$, the identification of $A_d$ as a subspace of the full Fock $\cF(H)$ space given in \eqref{eq:identification} allows us to define $X^J = (X^J_n)_{n=0}^\infty$ by
\[
X^J_n = H^{\otimes n} \ominus \{f_n \in J : f_n \textrm{ is } n\textrm{-homogeneous} \}. 
\]
Conversely, given a subproduct system $X = (X_n)_{n=0}^\infty$ with $X_1 = H$, we can define $J_X \subseteq A_d$ with the help of the orthogonal complement
\[
J_X = \overline{\sum_{n=0}^\infty H^{\otimes n} \ominus X_n}^{\|\cdot\|_\infty}.
\]

\begin{proposition}\label{prop:ideal_subps}
The correspondence $J \longmapsto X^J$ is a bijective correspondence between norm closed homogeneous ideals in $A_d$ and subproduct systems $X = (X_n)_{n=0}^\infty$ with $X_1 \subseteq H$. The inverse of $J \longmapsto X^J$ is given by $X \longmapsto J_X$. 
\end{proposition}
\begin{proof}
Indeed if $J$ is a homogeneous ideal, then for all homogeneous functions $f \in J$ and $g \in A_d$ we have $f g \in J$ and $g f \in J$. 
Using closedness of $J$, this translates to 
\[
(X^J_m)^\perp \otimes H^{\otimes n} \cup H^{\otimes m} \otimes (X^J_n)^\perp \subseteq (X^J_{m+n})^\perp = J_{m+n},
\]
which implies
\[
X^J_{m+n} \subseteq X^J_m \otimes X^J_n. 
\]
Conversely, if $X$ is a subproduct system, then every $n$-homogeneous element $f \in J_X$ satisfies, for every $m$-homogeneous polynomial $g \in \bF_d$, 
\[
gf \cong g \otimes f \in H^{\otimes m} \otimes X_n^\perp \subseteq X_{m+n}^\perp \subset J_X
\]
(and likewise from the right) and an approximation argument gives that $J_X$ is an ideal. 

Using Proposition \ref{prop:grading}, these two operations are seen to be mutual inverses of each other.
\end{proof}

\begin{example}\label{ex:symIdeal}
Let $X = \left(X_n\right)_{n=0}^\infty$ be the symmetric product system that was described in Example \ref{ex:sym}, that is, $X_n  =$ the symmetric product of $H$ with itself $n$-times where $H$ is some fixed Hilbert space. 
Then the corresponding ideal $J_X$ is the {\em commutator ideal} in $A_d$, which is equal to the closed ideal generated by the noncommutative polynomials $z_i z_j - z_j z_i$ for all $i,j=1, \ldots, d$ (details can be extracted from \cite[Proposition 2.4]{DavPitts98b}). 
Every subproduct subsystem of $X$ then corresponds to a homogeneous ideal in $A_d$ that contains the commutator ideal. 
When $d<\infty$, every such subproduct subsystem corresponds to a homogeneous ideal in the algebra $\bC[x_1, \ldots, x_d]$ of polynomials in $d$ commuting variables (see \cite[Section 2.3]{DRS11}). 
\end{example}

\begin{example}\label{ex:subshift}
In dynamical systems, {\em subshifts} are a class of symbolic dynamical systems that have been thoroughly investigated \cite{BSBook}. 
Given a subshift, one can construct a subproduct system $X$ that encodes the subshift (see \cite[Section 12]{ShaSol09}). 
The associated ideal $J_X$ is then the monomial ideal generated by the monomials (in noncommuting variables) corresponding to all {\em forbidden words} for the subshift. 
In this case, the shift operators on the $X$-Fock space are the basic ingredients in the construction of Matsumoto's subshift C*-algebras \cite{Mat97}; these subproduct systems and their algebras were studied extensively in \cite{KakSha19}. 
\end{example}

\section{Representations of tensor algebras}\label{sec:representation}

In this section we use the following notation: $X$ will be a subproduct system, $d = \dim X_1$, we will assume that there is some fixed orthonormal basis $(e_i)_{i \in \Lambda}$ for $H := X_1$, and we will write $S = (S_i)_{i \in \Lambda}$ for the $X$-shift restricted to the orthonormal basis, that is, $S_i = S(e_i)$.

\subsection{Universality of tensor algebras}\label{subsec:universality} 

The noncommutative disc algebra $\fA_d$ \cite{Pop96} is the tensor algebra of the full subproduct system that is the full product system over a Hilbert space $H$ with $\dim H = d$; see Example \ref{ex:product}. 
For this algebra we have a noncommutative functional calculus \cite{Pop95}: for every row contraction $T = (T_i)_{i \in \Lambda}$ with $|\Lambda|=d$, there exists a unital completely contractive homomorphism $f \mapsto f(T)$ from $\fA_d$ into the unital closed operator algebra $\ol{\alg}(T)$ generated by $T$, that maps $L_i = S(e_i)$ to $T_i$. 
When thinking of $\fA_d$ as $A_d = A(\fB_d)$, the noncommutative functional calculus when applied to $T \in \ol{\fB}_d$ corresponds to evaluating $f$ at the point $T$, and it allows us to think of $\fA_d$ (or $A_d$) as an algebra of functions defined on row contractions in {\em all} dimensions. 
This functional calculus gives the full shift $L$ the role of a universal row contraction.

The functional calculus allows us to say that a row contraction $T$ {\em annihilates} an ideal $J \triangleleft \fA_d$ if $f(T) = 0$ for all $f \in J$. 
The $X$-shift of a subproduct system $X$ is a universal row contraction that annihilates $J_X$. 
The following is known, and can be dug out of the literature (e.g. \cite{Pop06} or \cite{ShaSol09}), but unfortunately it does not appear in precisely the form we need. 
\begin{proposition}\label{prop:universal}
Let $X$ be a subproduct system, let $J = J_X$ be the corresponding closed homogeneous ideal in $\fA_d$, where $d = \dim X_1$, and let $(e_i)_{i \in \Lambda}$ be an orthonormal basis for $X_1$. 
For every unital completely contractive representation $\pi : \cA_X  \to B(K)$, the $d$-tuple $T = (T_i)_{i \in \Lambda} \in B(K)^\Lambda$ given by $T_i = \pi(S(e_i))$ is a row contraction that annihilates $J$. 
Conversely, every row contraction $T = (T_i)_{i \in \Lambda} \in B(K)^\Lambda$ that annihilates $J$ determines a unique unital completely contractive representation $\pi : \cA_X  \to B(K)$ given by $\pi(S_i) = T_i$.
\end{proposition}
\begin{proof}
This is known, but we take a moment to track down the references.
In the case $d<\infty$ this follows from Popescu's constrained dilation theory and von Neumann inequality \cite[Section 2]{Pop06} (see also the variant in \cite[Section 8]{ShaSol09}). 
To treat arbitrary $d$ one can note that the methods of the noncommutative Poisson transform extend to infinite dimensions, but we can alternatively appeal to the (also known) fact that the unital completely contractive representations of a tensor algebra are in one to one correspondence with so-called {\em representations} of the subproduct system $X$; see \cite[Corollary 2.16]{Vis11}, where this is proved in the greatest generality. 
The representations of a subproduct system $X$ are precisely the row contractions annihilating $J_X$ --- this is the content of \cite[Theorem 7.5]{ShaSol09}. Note carefully that while that theorem is stated for for arbitrary $d$, it is stated only for a polynomial ideal; however, the proof of that theorem does work for general closed ideals in $\fA_d$ with essentially no change. 
\end{proof}

\begin{theorem}\label{thm:quotient_compression}
Let $X$ be a subproduct system and let $J = J_X$ be the corresponding closed homogeneous ideal in $\fA_d$. 
Then the compression map $f \mapsto P_{\cF_X} f \big|_{\cF_X}$ has kernel $J$ and induces a completely isometric isomorphism of $\fA_d /J$ onto $\cA_X = P_{\cF_X} \fA_d \big|_{\cF_X}$. 
\end{theorem}
\begin{proof}
The compression of $L = (L_i)_{i \in \Lambda}$ is $S = (S_i)_{i\in \Lambda}$, and so we obtain a completely contractive homomorphism $\pi : \fA_d \to \cA_X$ that sends a polynomial $p(L)$ in $L$ to a polynomial $p(S)$ in $S$. 
The functional calculus maps $f \in \fA_d$ to $f(S) \in \cA_X$, and $f(S) = 0$ if and only if $f \in J$. 
Thus $\ker \pi = J$, and $\pi$ induces a completely contractive homomorphism $\ol{\pi} : \fA_d / J \to \cA_X$.
But now letting $\dot{L} = (\dot{L}_i)_{i \in \Lambda}$ be the image of $L$ in the quotient, we have that $\dot{L}$ is a row contraction that annihilates $J$. 
By Proposition \ref{prop:universal}, there is a completely contractive unital homomorphism $\cA_X \to \fA_d/J$ such that $S_i \mapsto \dot{L}_i$. 
It follows that $\ol{\pi}$ must be a completely isometric isomorphism. 
Finally, since the range of $\ol{\pi}$ is equal, on the one hand, to $\cA_X$, while on the other hand it is equal to the range of the compression map $\pi$, we have $\cA_X = P_{\cF_X} \fA_d \big|_{\cF_X}$.
\end{proof}

Alternatively, the above theorem could also be proved by using a commutant lifting approach or the distance formula $d(T,J) = \|P_{\cF_X} T \big|_{\cF_X}\|$ and its matrix valued variant (see \cite{AriPop00,DavPitts98a}, cf. \cite[Proposition 9.7]{SSS18}) and then one could deduce Proposition \ref{prop:universal} from the theorem. 

\subsection{Finite dimensional representability of tensor algebras}\label{subsec:fin_dim}

For an operator algebra we let $\operatorname{Rep}_{\fin}(\cB)$ denote the space of all unital completely contractive finite dimensional representations of $\cB$. 
Proposition \ref{prop:universal} implies that $\operatorname{Rep}_{\fin}(\cA_X)$ can be identified with the homogeneous nc variety
\[
\ol{V}(J_X) := \{X \in \ol{\fB}_d : f(X) = 0 \text{ for all } f \in J_X\}, 
\]
where the identification is given by 
\[
\operatorname{Rep}_{\fin}(\cA_X) \ni \rho \longleftrightarrow (\rho(S_i))_{i \in \Lambda}  \in \ol{V}(J_X),
\]
and
\[
\ol{V}(J_X) \ni X \longleftrightarrow \rho_X \in \operatorname{Rep}_{\fin}(\cA_X)
\]
where $\rho_X$ is the representation mapping $S_i$ to $X_i$. 
Every element $a \in \cA_X$ can be considered as a function on $\operatorname{Rep}_{\fin}(\cA_X)$ in the usual way $\hat{a}(\rho) = \rho(a)$, and the above identification gives $a$ as a function, which is also denoted by $\hat{a}$, on $\ol{V}(J_X)$. 
Let us call the map $a \mapsto \hat{a}$ the {\em restriction map}. 
Note that $\hat{a} \in A(\ol{V}(J_X))$, that is, $\hat{a}$ is a uniformly continuous nc function on $\ol{V}(J_X)$. 
Our goal in this section is to understand the extent to which the restriction map 
\begin{equation*}
  \mathcal{A}_X \to A(\overline{V}(J_X)), \quad  a \mapsto \hat{a},
\end{equation*}
is faithful.

\begin{definition}\label{def:rfd}
An operator algebra $\cB$ is said to be {\em residually finite dimensional} (or {\em RFD}) if for every $b \in M_n(\cB)$
\[
\|b\| = \sup\left\{\left\|\pi^{(n)}(b)\right\| : \pi \in \operatorname{Rep}_{\fin}(\cB)\right\}.
\]
\end{definition}

\begin{theorem}\label{thm:equivalent}
Let $J \triangleleft \fA_d$ be a homogeneous ideal and let $X = X^J$ the corresponding subproduct system. 
The restriction map $\mathcal{A}_X \to A(\overline{V}(J))$ is a complete quotient map.
Moreover, the following are equivalent: 
\begin{enumerate}
\item The ideal $J$ satisfies the Nullstellensatz. 
\item The restriction map $\cA_X \to A(\ol{V}(J))$ is injective. 
\item The restriction map $\cA_X \to A(\ol{V}(J))$ is a completely isometric isomorphism.
\item $\cA_X$ is residually finite dimensional. 
\end{enumerate}
\end{theorem}
\begin{proof}
By Theorem \ref{thm:quotient_compression}, the compression map induces a completely isometric isomorphism $\mathfrak{A}_d / J \cong \mathcal{A}_X$.
On the other hand, \cite[Proposition 9.7]{SSS18} shows that restriction to $\overline{V}(J)$ induces a completely isometric
isomorphism $\mathfrak{A}_d / I(\overline{V}(J)) \cong A(\overline{V}(J))$. Moreover, we have a commuting diagram
\begin{center}
\begin{tikzcd}
  \mathcal{A}_X \arrow[r] & A(\overline{V}(J)) \\
\mathfrak{A}_d / J \arrow[r] \arrow[u,"\cong"] & \mathfrak A_d/I(\overline{V}(J)) \ar[u,"\cong"],
\end{tikzcd}
\end{center}
where the top row is the restriction map and the bottom row is the natural quotient map.
In particular, we see that the restriction map $\mathcal{A}_X \to A(\overline{V}(J))$ is always a complete quotient map.

As for the equivalent conditions, we first observe that $I(V(J)) = I(\overline{V}(J))$.
Thus, $J$ satisfies the Nullstellensatz if and only if $J = I(\overline{V}(J))$.
By the previous paragraph, this happens if and only if the restriction map is injective.
Furthermore, since the restriction map is always a complete quotient map, injectivity is equivalent to being completely
isometric. Hence (1) $\Longleftrightarrow$ (2) $\Longleftrightarrow$ (3).
Finally, the equivalence $(3) \Longleftrightarrow (4)$ follows immediately from the identification $\operatorname{Rep}_{\fin}(\cA_X) \cong \ol{V}(J_X)$. 
\end{proof}

\begin{remark}
The algebra $A(\overline{V}(J))$ is a quotient of $\mathcal{A}_X$ and is always RFD.
  Moreover, every finite dimensional representation of $\mathcal{A}_X$
  factors through $A(\overline{V}(J))$ because of the identification $\operatorname{Rep}_{\fin}(\mathcal{A}_X) \cong \overline{V}(J)$.
  Thus, whenever $\pi: \cA_X \to \mathcal{B}$ is a complete quotient map onto another RFD algebra $\mathcal{B}$,
  then $\pi$ factors through a complete quotient map $A(\overline{V}(J)) \to \mathcal{B}$.
  In this sense, $A(\overline{V}(J))$ is the largest RFD quotient of $\mathcal{A}_X$.
\end{remark}

In \cite[Theorem 9.5]{SSS18} it was shown that under the assumption that $d< \infty$, every homogeneous closed ideal satisfies the Nullstellensatz. 
In the next theorem we record the fact that if an ideal $J$ is generated by homogeneous polynomials involving only finitely many variables, then $J$ satisfies the Nullstellensatz. 
The next theorem also gives another sufficient condition that works in the case $d = \infty$, namely that the ideal is generated by monomials. 
Monomial ideals are a special class of ideals, but it is worth recalling that subproduct systems associated with monomial ideals give rise to many operator algebras including Cuntz-Krieger and subshift C*-algebras \cite{KakSha19}. 

\begin{theorem}\label{thm:nullstz}
  If a closed homogeneous ideal in $J \triangleleft A_d$ is generated by monomials or if there exists a finite subset $\Lambda' \subset \Lambda$
  such that $J$ is generated by homogeneous polynomials in the variables $(z_i)_{i \in \Lambda'}$, then
\[
J = I(V(J)). 
\]
\end{theorem}
\begin{proof}

Clearly, $J \subseteq I(V(J))$. 
On the other hand, if $f \in A_d \setminus J$, then at least one of the homogeneous components $f_n \notin J$. 
It suffices to show that $f_n \notin I(V(J))$, and we now relabel $f_n$ as $f$ and remember that it is homogeneous of degree $n$. 
Now we consider the subproduct system $X = X_J$ and we form the space $\cE = X_0 \oplus X_1 \oplus \cdots \oplus X_n$. 
Let $1$ denote the copy of $1$ in $X_0 = \bC$. 
For any $r<1$, the compression $Y$ of $(r S(e_i))_{i \in \Lambda}$ to $\cE$ is a strict row contraction such that $f(Y) 1 \neq 0$ while $g(Y) = 0$ for all $g \in J$. 
Now if $\cE$ happened to be finite dimensional then we have $Y \in V(J)$, while $f(Y) \neq 0$, thus $f \notin I(V(J))$ and the proof would be complete (this is precisely how the proof for the case $d < \infty$ works). 
In general, we cannot yet conclude that $Y \in V(J)$, because the space $\cE$ that $Y$ is acting on might be infinite dimensional. 

We overcome this difficulty as follows. 
The function $f \notin J$ that we are considering might not be a polynomial, but it is the limit in the norm of homogeneous polynomials of degree $n$. 
Now let $\epsilon > 0$ be a such that $\epsilon < \frac{\|f(Y)1\|}{3}$ and find a homogeneous polynomial $p$ such that $\|p-f\|_\infty <\epsilon$. 

Assume first that $J$ is generated by monomials.
Let $(z_i)_{i \in \Lambda'}$ be the coordinate functions that appear in $p$, and let $\cF \subseteq \cE$ be the subspace spanned by $q(Y)1$ for all noncommutative polynomials $q$ in the variables $\{z_i\}_{i \in \Lambda'}$ of degree less than or equal to $n$. 
Let $Z = (Z_i)_{i \in \Lambda} $ be the tuple defined such that $Z_i$ is equal to the compression of $Y_i$ to $\cF$ for $i \in \Lambda'$ and $Z_i = 0_\cF$
for $i \in \Lambda \setminus \Lambda'$.

We claim that $Z \in V(J)$. 
First, $Z$ is a strict row contraction and it is acting on the finite dimensional space $\cF$. 
Next, for every monomial $g \in J$, either, $g$ involves only the variables $(z_i)_{i \in \Lambda'}$, in which case $g(Z) = g(Y)\big|_\cF = 0$; or else $g$ involves one of the variables $z_i$ for $i \in \Lambda \setminus \Lambda'$, in which case $g(Z) = 0$.

Finally, we show that $f(Z) \neq 0$, establishing that $f \notin I(V(J))$. 
For this, note that 
\begin{align*}
\|f(Z)1\| &\geq \|p(Z)1\| - \|p(Z)1 - f(Z)1\| \\
&> \|p(Y)1\| - \epsilon \\
&\geq \|f(Y)1\| - \|f(Y)1-p(Y)1\| - \epsilon \\
&>  \frac{\|f(Y)1\|}{3} > 0
\end{align*}
as required.

If $J$ is generated by homogeneous polynomials in the variables $(z_i)_{i \in \Lambda'}$ for a finite set $\Lambda' \subset \Lambda$,
enlarge $\Lambda'$ so that $p$ is also a polynomial in $(z_i)_{i \in \Lambda'}$
and define $\mathcal{F}$ and $Z$ as above. A similar argument then shows that $Z \in V(J)$, but $f(Z) \neq 0$, so that $f \notin I(V(J))$.
\end{proof}

\begin{example}\label{ex:symIdeal_inf}
Let $J$ be the commutator ideal in $A_d$ (see Example \ref{ex:symIdeal}) for $d = \aleph_0$. 
Then $J$ does not meet the requirements of Theorem \ref{thm:nullstz}. 
Nevertheless, it does satisfy the Nullstellensatz. 
One can show this directly by appealing to \cite[Proposition 2.4]{DavPitts98b}. 
Alternatively, note that in this case, the corresponding $X$-Fock space for $X = X^J$ is the symmetric Fock space. 
Therefore $\cA_X$ is an algebra of multipliers on a reproducing kernel Hilbert space, so it is residually finite dimensional. 
By Theorem \ref{thm:equivalent}, $I(V(J)) = J$. 
\end{example}

\subsection{Failure of the Nullstellensatz}\label{subsec:failure}

In this section we exhibit a closed homogeneous ideal $J \triangleleft \, \fA_d$ that does not satisfy the Nullstellensatz. 
Consequently, if $X = X^J$ is the corresponding subproduct system, then the tensor algebra $\cA_X$ is not RFD and cannot be represented as an algebra of uniformly continuous nc functions on $V(J)$. 

We will work with $d = \aleph_0$, and simply write $d = \infty$.
It will be natural to work in the setting of nc functions, thus we identify the Fock space $\cF(H)$ with the nc Drury-Arveson space $\cH^2_\infty$, and we identify $\fA_\infty$ with the subspace $A_\infty$ of $\cH^2_\infty$ consisting of uniformly continuous functions on $\fB_\infty$. 

\begin{proposition}\label{prop:Nullstellensatz_failure}
There exists a closed homogeneous ideal $J \subset \fA_\infty$ such that $J \subsetneq I(V(J))$.
\end{proposition}

\begin{proof}
Let $f(z) = \sum_{n=1}^\infty 2^{-n} z_n^2$, let $M = \cH^2_\infty(2) \ominus \mathbb{C} f$ and let $J$ be the closed ideal of $\fA_\infty$ generated by $M$. 
Notice that the degree $2$ part of $J$ is simply $J_2 = M$.
In particular, $f \notin J$.

We will show that $\cH^2_\infty(2) \subset I(V(J))$, which in particular implies that $f \in I(V(J))$.
To this end, we observe that
\begin{enumerate}
\item $z_i z_j \in J$ whenever $i \neq j$, and
\item $2 z_{n+1}^2 - z_n^2 \in J$ for all $n \in \mathbb{N}$.
\end{enumerate}
Suppose now that $T = (T_1,T_2,\ldots) \in V(J)$. Since the entries of $T$ are elements of a finite dimensional vector space, they must be linearly dependent. Thus, there exist $k \in \mathbb{N}$ and scalars $\lambda_j \in \mathbb{C}$, all but finitely many of which are zero, such that
\begin{equation*}
T_k = \sum_{j \neq k} \lambda_j T_j.
\end{equation*}
Multiplying this relation with $T_k$ from the left and using (1) above, we find that
\begin{equation*}
T_k^2 = \sum_{j \neq k} \lambda_j T_k T_j = 0.
\end{equation*}
Relation (2) above then implies that $T_n^2 = 0$ for all $n \in \mathbb{N}$. In combination with (1),
this implies that $\cH^2_\infty(2) \subset I(V(J))$, as desired.
\end{proof}

\begin{remark}
\begin{enumerate}[label=\normalfont{(\alph*)},wide]
\item 
The example constructed in the proof of Proposition \ref{prop:Nullstellensatz_failure}
is in fact commutative, as Relation (1) in the proof shows.
\item Let $J$ be the ideal constructed in the proof of Proposition \ref{prop:Nullstellensatz_failure}.
It follows from (1) and (2) in the proof that $\cH^2_\infty(3) \subset J$. Hence the $X$-shift associated with $J$ is jointly nilpotent of order $3$.
The proof shows that any row contractive tuple of matrices satisfying the relations in $J$ is nilpotent of order $2$.
\end{enumerate}
\end{remark}

\section{Classification by subproduct systems}\label{sec:classification}

\subsection{Reduction of the isomorphism problem to existence of graded maps}\label{subsec:reduction}

In this subsection we recall known results about the isomorphism problem for tensor algebras of subproduct systems, results which allow one to restrict attention to {\em graded} maps. 

Recall the maps $\Phi_n : \cA_X \to \cA_X^{(n)}$ from Proposition \ref{prop:grading}. 
\begin{definition}\label{def:graded_vp}
The {\em vacuum state} is the homomorphism $\Phi_0 \colon \cA_X \to \cA_X^{(0)} = \bC$. 
A homomorphism $\varphi \colon \cA_X \to \cA_Y$ is said to be {\em vacuum preserving} if $\varphi^*$ takes the vacuum state of $\cA_Y$ to the vacuum state of $\cA_X$ (here, $\varphi^*$ denotes the adjoint map $\varphi^* (f) = f \circ \varphi$), and it is said to be {\em graded} if $\varphi(\cA_X^{(n)}) \subseteq \cA_Y^{(n)}$ for all $n \in \mathbb{N}$.
\end{definition}

In \cite[Section 6]{DorMar14}, Dor-On and Markiewicz reduced a general form of Problem \ref{prob:iso} to the problem of whether the existence of an isomorphism implies the existence of a vacuum preserving isomorphism. 

\begin{proposition}[Propositions 6.12, 6.17, 6.18 and Corollary 6.13 in \cite{DorMar14}]\label{prop:TFAE}
Let $X$ and $Y$ be subproduct systems. 
The following are equivalent: 
\begin{enumerate}
\item There exists a vacuum preserving bounded isomorphism $\varphi \colon \cA_X \to \cA_Y$. 
\item There exists a vacuum preserving completely bounded isomorphism $\varphi \colon \cA_X \to \cA_Y$. 
\item There exists a similarity $W \colon X \to Y$. 
\end{enumerate}
Also, the following are equivalent: 
\begin{enumerate}
\item There exists a vacuum preserving isometric isomorphism $\varphi \colon \cA_X \to \cA_Y$. 
\item There exists a vacuum preserving completely isometric isomorphism $\varphi \colon \cA_X \to \cA_Y$. 
\item There exists an isomorphism $W \colon X \to Y$. 
\end{enumerate}
\end{proposition}
Let us say a few words about the proof of Proposition \ref{prop:TFAE}. 
We indicated right after Definition \ref{def:subps_similar} how a similarity/isomorphism between the subproduct systems gives rise to a completely bounded/isometric isomorphism between the tensor algebras. 
Conversely, assume first that we have a graded bounded isomorphism $\varphi \colon \cA_X \to \cA_Y$. 
In this case one can check that the ``restriction" of $\varphi$ to $\cA_X^{(n)} \cong X_n$ defines maps $V_n : X_n \to Y_n$, in other words
\[
V_n \colon x \in X_n \mapsto  S^{-1}(\varphi(S(x)) \in Y_n ,
\]
and that these maps assemble to form a similarity of subproduct systems \cite[Proposition 6.12]{DorMar14}. 
Furthermore, by Proposition \ref{prop:grading}, if $\varphi$ is isometric then $V$ is an isomorphism of subproduct systems. 

If $\varphi$ is not necessarily graded, but merely vacuum preserving, then one can show that $\varphi$ can be modified to a graded isomorphism $\tilde{\varphi}$ by ``dropping higher order terms", that is
\[
\tilde{\varphi}\big|_{\cA_X^{(n)}} = \Phi_n \circ \varphi\big|_{\cA_X^{(n)}} . 
\]
See Proposition 6.18 and 6.19 in \cite{DorMar14} (this modification is needed for the case of bounded isomorphism; a vacuum preserving {\em isometric} isomorphism is already graded, see \cite[Theorem 9.7]{ShaSol09}). 

Thus, the isomorphism problem will be settled once we show that the existence of an isomorphism implies the existence of a vacuum preserving isomorphism.

\begin{remark}
The reader might wonder whether the notions of ``isomorphic" and ``similar" subproduct systems are actually different. 
The answer is yes. 
The tensor algebras corresponding to commutative, radical ideals in the case $d < \infty$ (see Example \ref{ex:symIdeal}) have been completely classified in terms of the geometry of the zero sets of the ideals in \cite{DRS11} and \cite{Har12}, and by using this result (together with Theorem \ref{thm:TFAE}) it is easy to exhibit examples of similar but not isomorphic subproduct systems. 
For example, it follows that if the zero sets of the ideals $J_X$ and $J_Y$ are both equal to a union of two complex lines, then $X$ and $Y$ are similar, while $X$ and $Y$ are isomorphic only when one zero set is the image of the other zero set under a unitary. 

It is often technically more tractable to classify subproduct systems than to classify the operator algebras, and this is part of the motivation for the formulation of the isomorphism problem. 
For example, within the class of subproduct systems corresponding to subshifts (Example \ref{ex:subshift}), it was shown in \cite[Theorem 9.2]{KakSha19} that algebraic isomorphism of the tensor algebras is equivalent to completely isometric isomorphism of the tensor algebras; this is most readily seen by showing that, in this context, similarity of the subproduct systems implies isomorphism of the subproduct systems (see \cite[Remark 9.4]{KakSha19}). 
\end{remark}

\subsection{The disc trick}\label{subsec:disc_trick}

Let $X$ and $Y$ be two subproduct systems. We will show that the tensor algebras $\cA_X$ and $\cA_Y$ are isometrically isomorphic if and only if $X$ is isomorphic to $Y$, and that $\cA_X$ and $\cA_Y$ are completely boundedly isomorphic if and only if $X$ is similar to $Y$. 
As explained in Section \ref{subsec:reduction}, Proposition \ref{prop:TFAE} reduces our task to showing that if there exist an isometric or a completely bounded isomorphism $\varphi : \cA_X \to \cA_Y$, then there is a vacuum preserving one. 
Showing how isomorphisms give rise to vacuum preserving ones can be achieved using a technique that has come to be known as ``the disc trick" (see \cite{OrrBlog}), which becomes available after one shows that the induced map $\varphi^*$ between the character spaces sends a disc to a disc.
In \cite{DRS11} this was shown using ideas from elementary algebraic geometry, but it is not clear whether these methods extend to the infinite dimensional case. 
A different approach for showing that $\varphi^*$ maps a disc to a disc was introduced in \cite{Hartz17a} to treat certain {\em weighted} tensor algebras with finite dimensional fibers; this approach is very general and can be adapted to the infinite dimensional setting. 

Since the following two lemmas are of independent interest and are applicable in a wider scope, it will be convenient for us to somewhat change our notation as follows.
If $\mathcal{E}$ is a Hilbert space, we let $B_1(\mathcal{E})$ denote the open unit ball of $\mathcal{E}$.
A non-empty subset $A \subset B_1(\mathcal{E})$ is said to be {\em homogeneous} if $x \in A$
implies $\mathbb{C} x \cap B_1(\mathcal{E}) \subset A$ (cf. the definition of homogeneous set right after Definition \ref{def:hom_ideal}).
Clearly, every homogeneous set contains the origin. 
For every $\lambda \in \bT$, let $U_\lambda: \mathcal{E} \to \mathcal{E}$ denote the gauge transformation $z \mapsto \lambda z$. 
Clearly a homogeneous subset $A \subset B_1(\cE)$ is invariant under the gauge transformations. 
We require the following ad hoc definition.

\begin{definition}
Let $\cE,\cF$ be Hilbert spaces and let $A \subset \cE$ be a homogeneous set.
A map $F: A \to \mathcal{F}$ is said to be {\em holomorphic} if for all $x \in A \setminus \{0\}$ and all $y \in \mathcal{F}$, the function
\begin{equation*}
  \mathbb{D} \to \mathbb{C}, \quad t \mapsto \Big \langle F \Big( t \frac{x}{\|x\|} \Big) , y \Big\rangle,
\end{equation*}
is holomorphic. 
A {\em biholomorphism} is a bijective holomorphic map with holomorphic inverse.
\end{definition}

We begin with the following standard adaptation of the maximum modulus principle
and the Schwarz lemma,
which is a slight generalization of Lemmas 9.1 and 9.2 in \cite{Hartz17a}.
The proof carries over almost verbatim.

\begin{lemma}
  \label{lem:Schwarz}
  Let $A \subset B_1(\mathcal{E})$ be homogeneous and let $F: A \to \overline{B_1(\mathcal{F})}$
  be a holomorphic map.
  \begin{enumerate}[label=\normalfont{(\alph*)}]
    \item If $F$ is not constant, then $\|F(z)\| < 1$ for all $z \in A$.
    \item Suppose that $F(0) = 0$. Then $\|F(z)\| \le \|z\|$ for all $z \in A$.
      If equality holds for some $z \neq 0$, then
  \begin{equation*}
    F \Big( t \frac{z}{\|z\|} \Big) = t \frac{F(z)}{\|F(z)\|} \quad \text{ for all }  t \in \mathbb{D}.
  \end{equation*}
  \end{enumerate}
\end{lemma}

\begin{proof}
  In both cases, we may assume that $A \neq \{0\}$.

  (a) Suppose that $\|F(w)\| = 1$ for some $w \in A$.
  We first show that $F(0) = F(w)$, for which we may assume that $w \neq 0$.
  Then the holomorphic function
  \begin{equation*}
    \mathbb{D} \to \overline{\mathbb{D}}, \quad t \mapsto
    \Big \langle F \Big( t \frac{w}{\|w\|} \Big), F(w) \Big \rangle,
  \end{equation*}
  maps $\|w\|$ to $1$ and hence is identically $1$ by the maximum modulus principle.
  Thus, equality holds in the Cauchy--Schwarz inequality, so $F( t \frac{w}{\|w\|}) = F(w)$
  for all $t \in \mathbb{D}$, and in particular $F(0) = F(w)$.
  
  Next, if $z \in A \setminus \{0\}$ is arbitrary, we apply the maximum modulus principle and the equality case of the Cauchy--Schwarz inequality to
  \begin{equation*}
    \mathbb{D} \to \overline{\mathbb{D}}, \quad t \mapsto \Big \langle F \Big( t \frac{z}{\|z\|} \Big), F(0) \Big \rangle,
  \end{equation*}
  to find that $F(t \frac{z}{\|z\|}) = F(0)$ for all $t \in \mathbb{D}$; hence $F(z) = F(0)$.
  Therefore, $F$ is constant.

  (b) Let $z \in A \setminus \{0\}$ and assume that $F(z) \neq 0$.
  The Schwarz lemma shows that the function
  \begin{equation*}
    f: \mathbb{D} \to \overline{\mathbb{D}}, \quad
    t \mapsto \Big \langle F \Big( t \frac{z}{\|z\|} \Big), \frac{F(z)}{\|F(z)\|} \Big\rangle,
  \end{equation*}
  satisfies $|f(t)| \le |t|$ for all $t \in \mathbb{D}$. The first statement follows by taking $t = \|z\|$.

  If $\|F(z)\| = \|z\|$, then $f(\|z\|) = \|z\|$, hence $f$ is the identity on $\mathbb{D}$
  by the equality case in the Schwarz lemma. Since
  $\|F(t \frac{z}{\|z\|})\| \le |t|$ for all $t \in \mathbb{D}$ by the first part, we also have equality in the Cauchy--Schwarz inequality,
  so $F(t \frac{z}{\|z\|}) = t \frac{F(z)}{\|F(z)\|}$ for all $t \in \mathbb{D}$.
\end{proof}

The following result will allow the reduction to the vacuum preserving case.

\begin{lemma}
  \label{lem:disc_trick}
  Let $\mathcal{E},\mathcal{F}$ be Hilbert spaces, let $A \subset B_1(\mathcal{E})$ and $B \subset B_1(\mathcal{F})$
  be homogeneous sets and let $F: A \to B$ be a biholomorphism. Then there exist
  $\lambda, \mu \in \mathbb{T}$ such that
  \begin{equation*}
    F \circ U_\lambda \circ F^{-1} \circ U_\mu \circ F : A \to B
  \end{equation*}
  maps $0$ to $0$.
\end{lemma}

\begin{proof}
  We may assume that $F(0) \neq 0$, and define $b = F(0)$ and $a = F^{-1}(0)$.
  In the first step, we show that $F$ maps the disc $(\mathbb{C} a) \cap A \subseteq B_1(\mathcal{E})$ onto the disc
  $(\mathbb{C} b) \cap B \subseteq B_1(\mathcal{F})$.
  To this end, let
  \begin{equation*}
    f: \mathbb{D} \to B, \quad t \mapsto F \Big( t \frac{a}{\|a\|} \Big),
  \end{equation*}
  and let $\theta$ be a biholomorphic automorphism of $\mathbb{D}$ mapping $0$ to $\|a\|$ and vice versa.
  Lemma \ref{lem:Schwarz} (b), applied to $h = f \circ \theta$, shows that
  \begin{equation*}
    \|b\| = \|h(\|a\|)\| \le \|a\|.
  \end{equation*}
  By symmetry, $\|a\| \le \|b\|$, so equality holds. The second part of Lemma \ref{lem:Schwarz} (b)
  now shows that $h$ maps $\mathbb{D}$ onto $(\mathbb{C} b) \cap B_1(\mathcal{F})$. Hence $F$
  maps $(\mathbb{C} a) \cap B_1(\mathcal{E})$ onto $(\mathbb{C} b) \cap B_1(\mathcal{F})$,
  which completes the proof of the first step.

Now that we know that $F$ maps the disc $D_1 = (\mathbb{C} a) \cap A$ onto the disc $D_2 = (\mathbb{C} b) \cap B$, the known versions of the disc trick yield the result; see \cite[Proposition 4.7]{DRS11}, \cite[Lemmas 9.5 and 9.6]{Hartz17a} or \cite{OrrBlog}. 
For completeness we repeat the argument. 
Assume that $F(0) \neq 0$, for otherwise there is nothing to prove. 
Define
\[
C = \{U_\mu \circ F (0) : \mu \in \bT\} \subset D_2, 
\]
a circle centered at $0$ with radius $\|F(0)\|$. 
Then $F^{-1}(C) \subset D_1$ is a circle the passes through $0$, with $F^{-1}(0)$ in its interior. 
Rotate $F^{-1}(C)$ until its boundary hits $F^{-1}(0)$ and {\em voil\`a}--- we have found $\mu, \lambda \in \bT$ such that $F^{-1}(0) = U_\lambda \circ F^{-1} \circ U_\mu \circ F(0)$, or
\[
F \circ U_\lambda \circ F^{-1} \circ U_\mu \circ F(0) = 0.
\] 
Since $A$ and $B$ are homogeneous and $F : A \to B$ a biholomorphism, we have that $F \circ U_\lambda \circ F^{-1} \circ U_\mu \circ F : A \to B$ is a biholomorphism mapping $0$ to $0$, as required. 
\end{proof}

\subsection{Classification}\label{subsec:classification}

We let $\ol{\bB}_d$ denote the closed unit ball of a Hilbert space $H$ with $\dim H = d$. 
We shall consider some fixed orthonormal basis $\{e_i\}_{i \in \Lambda}$ for $H$, and use the identification $h \leftrightarrow \left( \langle h, e_i \rangle \right)_{i \in \Lambda}$ to think of elements in $\ol{\bB}_d$ as $d$-tuples of scalar row contractions. 
Thus, $\ol{\bB}_d$ can be considered as the first level $\ol{\fB}_d(1)$ of the closed nc unit ball $\ol{\fB}_d$.
Likewise, we let $\bB_d$ denote the corresponding open unit ball. 

If $J \subset A_d$ is an ideal, we let $\ol{Z}(J)$ denote the first level of the nc set $\ol{V}(J)$, that is
\begin{equation*}
\ol{Z}(J) = \{ z \in \overline{\mathbb{B}}_d: f(z) = 0 \text{ for all } f \in J \}
\end{equation*}
is the scalar vanishing locus of $J$ in the closed ball $\ol{\bB}_d$. 
We also write $Z(J) = \ol{Z}(J) \cap \mathbb{B}_d$.
We endow $\ol{Z}(J)$ with the (relative topology of the) weak topology.

If $\cB$ is a Banach algebra, we let $\cM(\cB)$ denote the character space of $\cB$, that is, the space of all nonzero homomorphisms of $\cB$ into $\bC$. 

\begin{proposition}
\label{prop:character_space}
Let $J \triangleleft A_d$ be a homogeneous ideal, let $X = X^J$ be the associated subproduct system
and let $\mathcal{A}_X$ be the tensor algebra of $X$. Then
\begin{equation*}
\mathcal{M}(\mathcal{A}_X) \to \ol{Z}(J), \quad \chi \mapsto (\chi(S_i))_{i \in \Lambda},
\end{equation*}
is a homeomorphism.
\end{proposition}

\begin{proof}
Let $\Phi$ denote the map in the statement. 
Since every character of an operator algebra is completely contractive, $\cM(\cA_X)$ is just the space of unital one dimensional completely contractive representations, thus $\Phi$ is a bijection by the discussion in the beginning of  Section \ref{subsec:fin_dim}. 
Moreover, $\Phi$ is continuous by definition
of the weak-$*$ topology and the fact that weak convergence on bounded sets is equivalent
to coordinatewise convergence. Since $\mathcal{M}(\mathcal{A}_X)$ is compact and the weak topology
is Hausdorff, it is a homeomorphism.
\end{proof}

We now prove that the existence of an isomorphism of a certain type between the tensor algebras implies the existence of a vacuum preserving isomorphism of the same type. 
As explained in Section \ref{subsec:reduction}, this will conclude the proof of Theorem \ref{thm:TFAE}.

\begin{theorem}
Let $\mathcal{A}_X$ and $\mathcal{A}_Y$ be (isometrically/completely boundedly) isomorphic tensor algebras
of subproduct systems.
Then there exists a vacuum preserving (isometric/completely bounded) isomorphism between $\mathcal{A}_X$ and $\mathcal{A}_Y$. 
\end{theorem}

\begin{proof}
Let $I,J$ be the homogeneous ideals associated with $X,Y$, respectively.
Let $\varphi: \mathcal{A}_X \to \mathcal{A}_Y$ be an isomorphism.
By Proposition \ref{prop:character_space}, the adjoint map $\varphi^* : \chi \mapsto \chi \circ \varphi$ can be regarded
as a homeomorphism $\varphi^*: \ol{Z}(J) \to \ol{Z}(I)$. 
Modulo the identification in Proposition \ref{prop:character_space},
the adjoint is given by
\begin{equation*}
\varphi^*(\xi) = (\varphi(S^X_i)(\xi))_{i \in \Lambda}. 
\end{equation*}
Indeed,  the $i$th component of $\varphi^*(\xi)$ is $S^X_i(\varphi^*(\xi)) = \varphi(S^X_i)(\xi)$, where we have written $S^X_i = S^X(e_i)$. 


From Section \ref{subsec:fin_dim}, we have the continuous restriction map
$\mathcal{A}_Y \to A(\overline{V}(J))$.
At the level of scalars, this shows that each element of $\mathcal{A}_Y$
restricts to a function on $Z(J)$ that is a uniform limit of polynomials, hence holomorphic in each disc contained in $Z(J)$ centered at the origin.
Hence $\varphi^*$ is holomorphic on the homogeneous set $Z(J)$.
In this setting, Lemma \ref{lem:Schwarz} (a) implies that $\varphi^*$ maps $Z(J)$ into $Z(I)$.
By symmetry, we find that $\varphi^*$ is a biholomorphism
from $Z(J)$ onto $Z(I)$.
Applying Lemma \ref{lem:disc_trick} with $A = Z(J), B = Z(I)$ and $F = \varphi^*$,
we obtain $\lambda,\mu \in \mathbb{T}$ so that
\begin{equation}\label{eq:composition}
\varphi^* \circ U_\lambda \circ (\varphi^{-1})^* \circ U_\mu \circ \varphi^*
\end{equation}
maps $0$ to $0$.

Given $\lambda \in \bT$ let $\Gamma_\lambda$ be the gauge automorphism on $\cA_X$ determined uniquely by $\Gamma_\lambda(S^X_i) = \lambda S^X_i$ (the existence of $\Gamma_\lambda$ is guaranteed by Proposition \ref{prop:universal}). 
We use the same notation for the gauge automorphism on $\cA_Y$. 
Note that 
\[
\Gamma_\lambda^* (\xi) (S^X_i) = \lambda S^X_i (\xi) = \lambda \xi_i  = S^X_i (\lambda\xi), 
\]
whence $\Gamma_\lambda^* = U_\lambda$. 
We find that \eqref{eq:composition} is the adjoint of the map
\begin{equation*}
\psi := \varphi \circ \Gamma_\mu \circ \varphi^{-1} \circ \Gamma_\lambda \circ \varphi: \mathcal{A}_X \to \mathcal{A}_Y ,
\end{equation*}
which is therefore a vacuum preserving isomorphism.
Finally, $\psi$ is isometric if $\varphi$ is, and $\psi$ is a complete isomorphism (completely bounded isomorphism with a completely bounded inverse) if $\varphi$ is.
\end{proof}


\bibliographystyle{amsplain}

\end{document}